\documentclass[a4paper]{amsart}
\usepackage{amssymb,latexsym,amsthm,amssymb,amsfonts,amsmath,amsopn,amstext,amscd,latexsym}
\numberwithin{equation}{section}
\usepackage{enumerate,xcolor,url,hyperref}
\usepackage[right]{lineno}

\newcommand{\bfC}{{\mathbf C}}
\newcommand{\bbF}{{\mathbb F}}

\newcommand{\Ker}{\mathop{\mathrm{Ker}}}
\newcommand{\Alt}{\mathop{\mathrm{Alt}}}
\newcommand{\Sym}{\mathop{\textrm{Sym}}}
\newcommand{\Aut}{\mathop{\mathrm{Aut}}}
\newcommand{\Irr}{\mathop{\mathrm{Irr}}}
\newcommand{\Tr}{\mathop{\mathrm{Tr}}}

\newcommand{\PSL}{\mathop{\mathrm{PSL}}}

\newcommand{\PGL}{\mathop{\mathrm{PGL}}}

\newcommand{\Sp}{\mathop{\mathrm{Sp}}}

\newcommand{\Sz}{\mathop{\mathrm{Sz}}}
\newcommand{\Ree}{\mathop{\mathrm{Ree}}}

\newcommand{\varep}{\varepsilon}

\newtheorem{theorem}{Theorem}[section]
\newtheorem{lemma}[theorem]{Lemma}
\newtheorem{proposition}[theorem]{Proposition}
\newtheorem{conjecture}[theorem]{Conjecture}

\newtheorem{claim}[theorem]{Claim}
\def\cent#1#2{{\bf C}_{{#1}}{{(#2)}}}
\def\norm#1#2{{\bf N}_{{#1}}{{(#2)}}}
\begin{document}

\title[Weak ERK-property]{An Erd\H{o}s-Ko-Rado theorem for finite $2$-transitive groups}

\author[K. Meagher]{Karen Meagher}
\address{ Department of Mathematics and Statistics,\newline
University of Regina, 3737 Wascana Parkway, S4S 0A4 Regina SK, Canada}\email{karen.meagher@uregina.ca}

\author[P. Spiga]{Pablo Spiga}
\address{
Dipartimento di Matematica e Applicazioni, University of Milano-Bicocca,\newline
Via Cozzi 55, 20125 Milano, Italy}\email{pablo.spiga@unimib.it}

\author[P. H. Tiep]{Pham Huu Tiep}
\address{
Department of Mathematics, University of Arizona, Tucson, AZ 85721-0089, USA}\email{tiep@math.arizona.edu}

\thanks{The first author is supported by NSERC. The third author was partially supported by the NSF grant DMS-1201374 and the Simons 
Foundation Fellowship 305247.\\ \noindent Address correspondence to Pablo Spiga. (pablo.spiga@unimib.it)}


\begin{abstract}
  We prove an analogue of the classical Erd\H{o}s-Ko-Rado theorem for
  intersecting sets of permutations in finite $2$-transitive
  groups. Given a finite group $G$ acting faithfully and
  $2$-transitively on the set $\Omega$, we show that an intersecting
  set of maximal size in $G$ has cardinality $|G|/|\Omega|$. This
  generalises and gives a unifying proof of some similar recent
  results in the literature.
\end{abstract}

\maketitle

\section{General results}\label{generalresults}

The Erd\H{o}s-Ko-Rado theorem~\cite{ErKoRa} determines the cardinality
and also describes the structure of a set of maximal size of
intersecting $k$-subsets from $\{1,\ldots,n\}$. The theorem shows that
provided that $n > 2k$, a set of maximal size of intersecting
$k$-subsets from $\{1,\dots, n\}$ has cardinality ${n-1\choose k-1}$
and is the set of all $k$-subsets that contain a common fixed
element. (For our work it is useful to emphasize that this theorem
consists of two distinct parts: the first part determines the maximal
size of intersecting $k$-subsets; the second part classifies the sets
attaining this maximum.) Analogous results hold for many other
combinatorial and algebraic objects other than sets, and in this paper
we are concerned with an extension of the Erd\H{o}s-Ko-Rado theorem to
permutation groups.

Let $G$ be a permutation group on $\Omega$. A subset $S$ of $G$ is
said to be \emph{intersecting} if, for every $g,h\in S$, the permutation
$gh^{-1}$ fixes some point of $\Omega$ (note that this implies that
$\alpha^g = \alpha^h$, for some $\alpha \in \Omega$). As with the
Erd\H{o}s-Ko-Rado theorem, in this context we are interested in finding the
cardinality of an intersecting set of maximal size in $G$ and possibly
classifying the sets that attain this bound. 

The main theorem of this paper answers the first question for $2$-transitive groups.
\begin{theorem}\label{main}
  Let $G$ be a finite $2$-transitive group on the set $\Omega$. An
  intersecting set of maximal size in $G$ has cardinality
  $|G|/|\Omega|$.
\end{theorem}

Before giving some specific comments on Theorem~\ref{main} (which
ultimately relies on the classification of finite simple groups and on
some detailed analysis of the representation theory of some Lie type
groups), we give some historical background on this area of research.

\subsection{Erd\H{o}s-Ko-Rado-type theorems for permutation groups} 

Possibly the most interesting permutation group and the most
intriguing combinatorial object is the finite symmetric group
$\Sym(n)$ of degree $n$.  Here, the natural extension of the
Erd\H{o}s-Ko-Rado theorem for $\Sym(n)$ was independently proved
in~\cite{CaKu} and~\cite{LaMa}.  These papers, using different
methods, showed that every intersecting set of ${\Sym(n)}$ has
cardinality at most the ratio
$|\Sym(n)|/|\{1,\ldots,n\}|=(n-1)!$. They both further showed that the
only intersecting sets meeting this bound are the cosets of the
stabiliser of a point. The same result was also proved in~\cite{GoMe}
using the character theory of $\Sym(n)$.

Despite the exact characterization of the largest
intersecting sets for $\Sym(n)$, Theorem~\ref{main} cannot be
strengthened to include such a characterization for all 2-transitive
groups.  There are various $2$-transitive groups having intersecting
sets of size $|G|/|\Omega|$ which are not the cosets of the stabiliser
of a point, see Conjecture~$1.2$ below or~\cite{KaPa,KaPa2} for some
examples. Currently, it is not clear to what extent
(that is, for which families of permutation groups) the complete
analogue of the Erd\H{o}s-Ko-Rado theorem holds.

Recently there have been many papers proving that the natural
extension of the Erd\H{o}s-Ko-Rado theorem holds for specific
permutation groups $G$ (see~\cite{AhMeAlt,
  Ellis, MR2302532,KaPa,KaPa2,MR2419214}) and there are also two papers, \cite{AhMe}
and \cite{AhMetrans}, that consider when the natural extension of the
Erd\H{o}s-Ko-Rado theorem holds for transitive and $2$-transitive
groups. Again, this means asking if the largest intersecting sets in
$G$ are the cosets in $G$ of the stabiliser of a point. Typically, a
permutation group may have intersecting sets of size larger than the
size of the stabiliser of a point, let alone hope that every such
intersecting set is the coset of the stabiliser of a point.  However,
a behaviour very similar to $\Sym(n)$ is offered by $\PGL_2(q)$ in its
natural action on the projective line~\cite[Theorem~$1$]{KaPa}; the
intersecting sets of maximal size in $\PGL_2(q)$ are exactly the
cosets of the stabiliser of a point. It is not hard to see that in
projective general linear groups of dimension greater than $2$ there
are maximum intersecting sets that are not the cosets of the
stabiliser of a point (for instance, the cosets of the stabiliser of a
hyperplane).  This lead the first and the second author to pose the
following conjecture:

\begin{conjecture}[{\cite[Conjecture~$2$]{KaPa}}] The intersecting sets
  of maximal size in $\PGL_{n}(q)$ acting
  on the points of the projective space are exactly the
  cosets of the stabiliser of a point and the cosets of the stabiliser
  of a hyperplane.
\end{conjecture}
This conjecture has been settled only for $n\in \{1,2\}$ in~\cite{KaPa,KaPa2}.

\subsection{A conjecture inspired by Theorem~\ref{main}}

As far as we are aware, Theorem~\ref{main} is the first result that
holds for a very rich family of permutation groups; thus far, most
Erd\H{o}s-Ko-Rado-type of theorems for permutation groups have been
proved for rather specific families. In this paper we aim for a more
general result.

For some $2$-transitive groups we prove a stronger version of
Theorem~\ref{main}, which we now explain. Let $V$ be the vector space
over $\mathbb{C}$ having a basis $(g)_{g\in G}$ indexed by the
elements of $G$, thus $V$ is the underlying vector space of the group
algebra $\mathbb{C}G$. Given a subset $S$ of $G$, we write
$\chi_S=\sum_{s\in S}s$ for the characteristic vector of $S$.
Set 
\[
W=\langle \chi_S\mid S\textrm{ coset of the stabiliser of a
  point in }G\rangle.
\]

For some $2$-transitive groups $G$, we show that if $S$ is an
intersecting set of maximal size, then $\chi_S\in V$. We actually dare
to state the following conjecture (which we like to think as the
algebraic analogue of the combinatorial Erd\H{o}s-Ko-Rado theorem).

\begin{conjecture}\label{newconjecture}
  Let $G$ be a finite $2$-transitive group and let $V$ be the subspace
  of the group algebra $\mathbb{C}G$ spanned by the characteristic
  vectors of the cosets of the point stabilisers. If $S$ is an
  intersecting set of maximal cardinality in $G$, then the
  characteristic vector of $S$ lies in $V$.
\end{conjecture}

\subsection{Comments on Theorem~\ref{main} and structure of the paper}
The proof of Theorem~\ref{main} depends upon the classification of the
finite $2$-transitive groups and hence on the classification of the
finite simple groups.

In Section~\ref{algebraicgraph}, we show that the problem of
determining the intersecting sets of a transitive group $G$ is
equivalent to the problem of determining the independent sets in the
Cayley graph over $G$ with connection set the derangements of $G$. We
then show some elementary results showing that the latter problem is
related to the character theory of $G$ and we prove some lemmas that
will be useful when $G$ is $2$-transitive.

In Section~\ref{reductios}, we reduce the proof of Theorem~\ref{main}
to the case that $G$ is one of the following non-abelian simple
groups: $\PSL_n(q)$, $\Sp_{2n}(2)$, $\mathrm{PSU}_3(q)$, Suzuki groups
$\Sz(q)$, Ree groups $\Ree(q)$ and the sporadic Higman-Sims group
$HS$. The proof of Theorem~\ref{main} for $\Sz(q)$, $\Ree(q)$, $HS$
and $\mathrm{PSU}_3(q)$ is in
Section~\ref{sec:suzuki},~\ref{sec:reegroups},~\ref{sec:HigmanSims}
and~\ref{sec:PSU} respectively and depends heavily on the character
table of these groups (unitary groups are by far the hardest case
here). Finally the proof of Theorem~\ref{main} for $\PSL_n(q)$ and
$\Sp_{2n}(2)$ is in Section~\ref{sec:PSL} and~\ref{Sec:symplectic}
respectively and depends on the theory of Weil representations of
these groups.

\section{Some algebraic graph theory}\label{algebraicgraph}

The problem of determining the intersecting sets of a permutation
group $G$ can be formulated in graph-theoretic terminology.  We denote
by $\Gamma_G$ the \emph{derangement graph} of $G$: the vertices of
this graph are the elements of $G$ and the edges are the unordered pairs
$\{g,h\}$ such that $gh^{-1}$ is a \emph{derangement}, that is,
$gh^{-1}$ fixes no point. Now, an intersecting set of $G$ is simply an
\emph{independent set} or a \emph{coclique} of $\Gamma_G$, and
(similarly) the classical Erd\H{o}s-Ko-Rado theorem translates into a
classification of the independent sets of maximal cardinality of the
Kneser graphs.

Since the right regular representation of $G$ is a subgroup of the
automorphism group of $\Gamma_G$, we see that $\Gamma_G$ is a
\textit{Cayley graph}. Namely, if $\mathcal{D}$ is the set of
derangements of $G$, then $\Gamma_G$ is the Cayley graph on $G$ with
connection set $\mathcal{D}$.
Clearly, $\mathcal{D}$ is a union of $G$-conjugacy classes, so
$\Gamma_G$ is a \emph{normal} Cayley graph.

As usual, we simply say that the complex number $\xi$ is an
\emph{eigenvalue} of the graph $\Gamma$ if $\xi$ is an eigenvalue of
the adjacency matrix of $\Gamma$.  We use $\Irr(G)$ to denote the
\emph{irreducible complex characters} of the group $G$ and given
$\chi\in\Irr(G)$ and a subset $S$ of $G$ we write 
\[
\chi(S)=\sum_{s\in S}\chi(s).
\]
In the following lemma we recall that the eigenvalues of a normal
Cayley graph on $G$ are determined by the irreducible complex
characters of $G$. This a well-known result, we refer the reader to
Babai's work~\cite{Ba} and a proof is also given in~\cite[Section
11.12]{EKRbook}.

\begin{lemma}\label{eigenvalues}
  Let $G$ be a permutation group on $\Omega$ and let $\mathcal{D}$ be
  the set of derangements of $G$. The spectrum of the graph $\Gamma_G$
  is $\{\chi(\mathcal{D})/\chi(1)\mid \chi\in \Irr(G)\}$. Also, if
  $\tau$ is an eigenvalue of $\Gamma_G$ and $\chi_1,\ldots,\chi_s$ are
  the irreducible characters of $G$ such that
  $\tau=\chi_i(\mathcal{D})/\chi_i(1)$, then the dimension of the
  $\tau$-eigenspace of $\Gamma_G$ is $\sum_{i=1}^s\chi_i(1)^2.$
\end{lemma}
For simplicity, we will write $\lambda(\chi)$ for the eigenvalue of
$\Gamma_G$ afforded by the irreducible character $\chi$, that is,
\begin{equation}\label{equation:-10}
\lambda(\chi)=\frac{\chi(\mathcal{D})}{\chi(1)}=\frac{1}{\chi(1)}\sum_{g\in\mathcal{D}}\chi(g).
\end{equation}

The next result is the well-known \emph{ratio-bound} for independent
sets in a graph; for a proof tailored to our needs see for example~\cite[Lemma~$3$]{KaPa}.

\begin{lemma}\label{lemma:7}
Let $G$ be a permutation group, let $\tau$ be the minimum eigenvalue of
$\Gamma_G$, let $d$ be the valency of $\Gamma_G$ and let $S$ be an independent set of $\Gamma_G$.  Then
\[
\frac{|S|}{|G|}\leq \left(1-\frac{d}{\tau}\right)^{-1}.
\] 
If the equality is met, then $\chi_S-\frac{|S|}{|G|}\chi_{G}$ is an
eigenvector of $\Gamma_G$ with eigenvalue $\tau$.
\end{lemma}

At this point it is worthwhile to give a hint on (a simplified version
of) one of our main strategies in proving Theorem~\ref{main}. Let $G$
be a $2$-transitive group on $\Omega$ and let $\pi$ be its permutation
character. Write $\pi=\chi_0+\psi$, where $\chi_0$ is the principal
character of $G$. Observe that $\psi$ is an irreducible character of
$G$ because $G$ is $2$-transitive.

We will show that in many cases the minimal eigenvalue of $\Gamma_G$
is afforded by the irreducible complex character $\psi$, that is, the
minimal eigenvalue of $\Gamma_G$ is $\lambda(\psi)$, and hence this
minimum is
\[
\lambda(\psi)=\frac{1}{\psi(1)}\sum_{g\in\mathcal{D}}\psi(g)=\frac{1}{|\Omega|-1}\sum_{g\in
  \mathcal{D}}-1=-\frac{|\mathcal{D}|}{|\Omega|-1}.
\]
In this case, using Lemma~\ref{lemma:7}, we get that an independent set of
$\Gamma_G$ has cardinality at most 
$$|G|\left(1-\frac{|\mathcal{D}|}{\lambda(\psi)}\right)^{-1}=|G|\left(1-\frac{|\mathcal{D}|}{-|\mathcal{D}|/(|\Omega|-1)}\right)^{-1}=\frac{|G|}{|\Omega|}$$
and hence Theorem~\ref{main} holds for $G$.

Actually, sometimes a slightly stronger version of this strategy
works, which allows us to prove Conjecture~\ref{newconjecture} for
these cases.  In fact, assume that the minimal eigenvalue of
$\Gamma_G$ is $-|\mathcal{D}|/(|\Omega|-1)$ and that $\psi$ is the
unique irreducible character of $G$ affording this eigenvalue. Then,
combining Lemmas~\ref{eigenvalues} and~\ref{lemma:7}, we see that
$$W=\langle \chi_S\mid S\textrm{ an independent set of }\Gamma_G\textrm{ of cardinality }|G|/|\Omega|\rangle$$
is a subspace of $\mathbb{C}G$ with dimension $\chi_0(1)^2+\psi(1)^2=1+(|\Omega|-1)^2$. Now,~\cite[Proposition~$3.2$]{KaPa} and also~\cite{AhMe} show that 
$$V=\langle \chi_S\mid S\textrm{ coset of the stabiliser of a point}\rangle$$
has also dimension $1+(|\Omega|-1)^2$. Since $W\leq V$, we get $W=V$
and hence Conjecture~\ref{newconjecture} holds for $G$.

For future reference we highlight the argument in the previous paragraph in the following:
\begin{lemma}\label{lemma:8}
  Let $G$ be a $2$-transitive group on $\Omega$ and let $\mathcal{D}$
  be the set of derangements of $G$. If the minimum eigenvalue of
  $\Gamma_G$ is $-|\mathcal{D}|/(|\Omega|-1)$ and if $G$ has a unique
  irreducible character realising this minimum, then both
  Theorem~$\ref{main}$ and Conjecture~$\ref{newconjecture}$ hold for
  $G$.
\end{lemma}

The following will prove useful in various occasions.

\begin{lemma}\label{critical}
  Let $G$ be a $2$-transitive group on $\Omega$, let $\mathcal{D}$ be
  the set of derangements of $G$ and let $\pi=\chi_0+\psi$ be the
  permutation character of $G$ with $\chi_0$ the principal character
  of $G$. Let $\chi^*$ be an irreducible complex character of $G$
  with $\chi^*\neq \psi$. If $\lambda(\chi^*)$ is the minimum
  eigenvalue of $\Gamma_G$, then
$$\chi^*(1)\leq (|\Omega|-1)\sqrt{\frac{|G|}{|\mathcal{D}|}-2}.$$
\end{lemma}
\begin{proof}
  Let $A$ be the adjacency matrix of $\Gamma_G$, so $A$ is the
  $|G|\times |G|$-matrix with rows and columns labelled by the
  elements of $G$ defined by
\[
A_{g,h}=
\begin{cases}
1&\textrm{if }gh^{-1}\in\mathcal{D},\\
0&\textrm{if }gh^{-1}\notin\mathcal{D}.
\end{cases}
\]

Given a square matrix $X=(x_{i,j})_{i,j}$, we denote by
$\Tr(X)=\sum_{i}x_{i,i}$ the trace of $X$. A standard counting
argument gives
\begin{eqnarray*}\Tr(A^2)&=&\sum_{g\in G}(A^2)_{g,g}=\sum_{g\in G}\sum_{h\in G}A_{g,h}A_{h,g}\\
&=&\sum_{g\in G}|\{h\in G\mid gh^{-1}\in \mathcal{D}\}|=\sum_{g\in G}|\mathcal{D}|=|G||\mathcal{D}|.
\end{eqnarray*}
Now, the matrix $A$ (and hence $A^2$) can be diagonalised over $\mathbb{C}$ and it follows immediately from Lemma~\ref{eigenvalues} that
\begin{equation}\label{eq:1hh}
|G||\mathcal{D}|=\Tr(A^2)=\sum_{\chi\in \Irr(G)}\chi(1)^2\left(\frac{\chi(\mathcal{D})}{\chi(1)}\right)^2=\sum_{\chi\in \Irr(G)}\chi(\mathcal{D})^2.
\end{equation}

Since $\lambda(\chi^*)$ is the minimum eigenvalue of $\Gamma_G$, we have 
\[
\lambda(\chi^*)\leq \lambda(\psi)=-|\mathcal{D}|/(|\Omega|-1)<0
\]
and hence $\chi^*\neq \chi_0$.

All the terms in the summation in the right-hand side of
Eq.~\eqref{eq:1hh} are positive, so we get that
$|G||\mathcal{D}|\geq\chi_0(\mathcal{D})^2+\psi(\mathcal{D})^2+
\chi^*(\mathcal{D})^2$. Since
$\chi_0(\mathcal{D})=-\psi(\mathcal{D})=|\mathcal{D}|$, this becomes
$\chi^*(\mathcal{D})^2\leq
|G||\mathcal{D}|-2|\mathcal{D}|^2$. Thus
\begin{eqnarray}\label{eq:2hh}
|\lambda(\chi^*)|&=&\frac{|\chi^*(\mathcal{D})|}{\chi^*(1)}\leq\frac{|\mathcal{D}|}{\chi^*(1)}\sqrt{\frac{|G|}{|\mathcal{D}|}-2}=\frac{|\mathcal{D}|}{\psi(1)}\frac{\psi(1)}{\chi^*(1)}\sqrt{\frac{|G|}{|\mathcal{D}|}-2}
\\\nonumber
&=&|\lambda(\psi)|\frac{\psi(1)}{\chi^*(1)}\sqrt{\frac{|G|}{|\mathcal{D}|}-2}=|\lambda(\psi)|\frac{|\Omega|-1}{\chi^*(1)}\sqrt{\frac{|G|}{|\mathcal{D}|}-2}.
\end{eqnarray}
Now, $\lambda(\chi^*)\leq \lambda(\psi)$ and hence $|\lambda(\chi^*)|\geq |\lambda(\psi)|$. Thus the proof follows from  Eq.~\eqref{eq:2hh}. 
\end{proof}

Incidentally, the proof of Lemma~\ref{critical} shows that
$|G|/|\mathcal{D}|\geq 2$ and hence at most half of the elements of a
$2$-transitive group are derangements, this observation is a special
case of the main result in~\cite{GIS}.

Some suspicious reader, misguided by the comments preceding
Lemma~\ref{lemma:8}, might think that the hypothesis of
Lemma~\ref{critical} are never satisfied. In fact, by wishful
thinking, one tends to hope that
$\lambda(\psi)=-|\mathcal{D}|/(|\Omega|-1)$ is always the minimum
eigenvalue of $\Gamma_G$. However this is not the case. The simplest
example is given by the $2$-transitive action of the Higman-Sims group
$HS$ of degree $176$: here $\psi(1)=176-1=175$ and
$\lambda(\psi)=-79806$, but there exists an irreducible character
$\chi$ of degree $22$ with $\lambda(\chi)=-118650$.

In Sections~\ref{sec:HigmanSims},~\ref{sec:PSU} and~\ref{Sec:symplectic}, the Higman-Sims
group, $\mathrm{PSU}_3(q)$ and $\Sp_{2n}(2)$ are dealt with using a generalised
version of Lemma~\ref{lemma:7}, which we now discuss. Given a graph
$\Gamma$, a square matrix $A$ with rows and columns indexed by the
vertices of $\Gamma$ is said to be a \textit{weighted adjacency
  matrix} of $\Gamma$ if $A_{u,v}= 0$ whenever $u$ and $v$ are not
adjacent vertices. The ratio-bound (a.k.a.  Lemma~\ref{lemma:7}) holds
not just for the adjacency matrix, but also for a weighted adjacency
matrix (see~\cite[Section 2.4]{EKRbook} for a proof).

\begin{lemma}\label{weightedratio}
  Let $G$ be a permutation group and let $A$ be a weighted adjacency
  matrix of $\Gamma_G$. Let $d$ be the largest eigenvalue  and let $\tau$
be   the least eigenvalue of $A$. If $S$ is an independent set in $\Gamma_G$, then
\[
\frac{|S|}{|G|} \leq \left( 1 -\frac{d}{\tau} \right)^{-1}.
\]
\end{lemma}

For the scope of this paper, we only consider very special types of
weighted adjacency matrices. Let $G$ be a permutation group, let
$\mathcal{D}$ be the set of derangements of $G$ and let
$\mathcal{C}_1,\ldots,\mathcal{C}_\ell$ be the $G$-conjugacy classes
contained in $\mathcal{D}$.  Given ${\bf
  a}=(a_1,\ldots,a_\ell)\in\mathbb{C}^\ell$, define the \textit{${\bf
    a}$-weighted adjacency matrix} of $\Gamma_G$ to be the matrix $A$
with rows and columns indexed by elements of $G$ and
\[A_{\sigma,\pi}=
\begin{cases}a_i& \textrm{if }\sigma^{-1}\pi \in \mathcal{C}_i,\\
0&\textrm{otherwise}.
\end{cases}
\]
It is implicit in the work of Babai in~\cite{Ba} (see
also~\cite[Section 11.12]{EKRbook}) that the eigenvalues of $A$ are determined by the
irreducible complex characters of $G$. In fact, as in
Eq.~\eqref{equation:-10}, given an irreducible character $\chi$ of
$G$, the eigenvalue of the ${\bf a}$-weighted adjacency matrix
determined by $\chi$ is
\begin{align}\label{eq:weightedevalues}
\lambda(\chi,{\bf a}) = \frac{1}{\chi(1)}\sum_{i=1}^\ell \left( a_i \sum_{x\in \mathcal{C}_i}\chi(x) \right)
                       =\frac{1}{\chi(1)} \sum_{i = 1}^\ell a_i |\mathcal{C}_i| \chi(x_i),
\end{align}
where $(x_i)_{i\in \{1,\ldots,\ell\}}$ is a set of representatives of
the $G$-conjugacy classes $ (C_i)_{i\in \{1,\ldots,\ell\}}$. Clearly,
when ${\bf a}=(1,\ldots,1)$, we recover the eigenvalues of the
adjacency matrix of $\Gamma_G$.

For later reference it is worth to point out that, if the weights
$a_i$ are non-negative real numbers, then the principal character will
give the largest eigenvalue of the ${\bf a}$-weighted adjacency
matrix. (This fact follows easily from Eq.~\eqref{eq:weightedevalues}
and from the inequality $|\chi(g)|\leq \chi(1)$, which is valid for
every $g\in G$ and for every irreducible character $\chi$.)

We conclude this section proving an analogue of Lemma~\ref{critical}.
\begin{lemma}\label{weightedcritical}
  Let $G$ be a $2$-transitive group on $\Omega$, let $\mathcal{D}$ be
  the set of derangements of $G$, let
  $\mathcal{C}_1,\ldots,\mathcal{C}_\ell$ be the $G$-conjugacy classes
  contained in $\mathcal{D}$ and let $\pi=\chi_0+\psi$ be the
  permutation character of $G$ with $\chi_0$ the principal character
  of $G$.
  
  Let $a_1,\ldots,a_\ell$ be non-negative real numbers not all equal
  to zero and set ${\bf a}=(a_1,\ldots,a_\ell)$, and let $A$ be the
  ${\bf a}$-adjacency matrix of $\Gamma_G$.
  
  If $\lambda(\chi^*,{\bf a})$ is the minimum eigenvalue of $A$
  for some irreducible character $\chi^*$ with $\chi^*\neq
  \psi$, then
\[
\chi^*(1) \leq (|\Omega|-1) \sqrt{|G| \frac{\sum_{i=0}^{\ell} a_i^2 |\mathcal{C}_i|}{(\sum_{i = 1}^\ell a_i |\mathcal{C}_i|)^2} -2 }.
\]
\end{lemma}
\begin{proof}
  The proof is similar to the proof of Lemma~\ref{critical} and hence
  we omit most of the details. Eq.~\eqref{eq:weightedevalues} yields
\[
\lambda(\chi_0,{\bf a}) =\sum_{i = 1}^\ell a_i |\mathcal{C}_i|, \qquad
\lambda(\psi,{\bf a}) = \frac{-1}{|\Omega|-1} \sum_{i = 1}^\ell a_i |\mathcal{C}_i|.
\]

Similar to the proof of Lemma~\ref{critical},
\[
\Tr(A^2) = |G| \sum_{i=0}^{\ell} a_i^2 |\mathcal{C}_i|.
\]
The trace of $A^2$ is also equal to the sum of the squares of the eigenvalues of $A$.
Thus, just as in the proof of Lemma~\ref{critical}, this implies that
\[
\chi^*(1)^2 \lambda(\chi^*,{\bf a})^2 \leq  |G| \sum_{i=0}^{\ell} a_i^2 |\mathcal{C}_i| - 2\left(\sum_{i = 1}^\ell a_i |\mathcal{C}_i|\right)^2.
\]
Since $\lambda(\chi^*,{\bf a})$ is the minimum eigenvalue of $A$, we have $\lambda({\chi^*,{\bf a}})\leq \lambda(\psi,{\bf a})<0$ and hence $|\lambda(\psi,{\bar a})|\leq |\lambda(\chi^*,{\bf a})|$. Now, following the last  paragraph of the proof of Lemma~\ref{critical}, we obtain 
\[
\chi^*(1) \leq (|\Omega|-1) \sqrt{|G| \frac{\sum_{i=0}^{\ell} a_i^2 |\mathcal{C}_i|}{(\sum_{i = 1}^\ell a_i |\mathcal{C}_i|)^2} -2 }.
\]
\end{proof}


\section{Three reductions}\label{reductios}

Let $G$ be a finite $2$-transitive group on $\Omega$. A $1911$
celebrated theorem of Burnside shows that either $G$ contains a
regular subgroup or $G$ is almost simple,
see~\cite[Theorem~$4.1$B]{DM}. In the first case, $G$ contains a
subgroup $H$ with $H$ transitive on $\Omega$ and with every
non-identity element of $H$ being a derangement on $\Omega$. In
particular, $H$ is a clique of cardinality $|\Omega|$ in the
derangement graph $\Gamma_G$. It is then relevant for
Theorem~\ref{main} the following well-known lemma (usually referred to
as the clique-coclique bound, see for example~\cite[Section 2.1]{EKRbook}).

\begin{lemma}\label{cliquecoclique}
  Let $\Gamma$ be a finite graph of order $n$ having a group of
  automorphisms acting transitively on its vertices. Let $C$ be a
  clique of $\Gamma$ and let $S$ be a coclique of $\Gamma$. Then
  $|C||S|\leq n$ with equality if and only if $|C\cap S^g|=1$ for
  every $g\in G$.
\end{lemma}

In particular, Theorem~\ref{main} follows immediately from
Lemma~\ref{cliquecoclique} for $2$-transitive groups having a regular
subgroup. Therefore, for the rest of this paper, we assume that $G$ is
almost simple, that is, $G$ contains a non-abelian simple normal
subgroup $S$ and $G$ acts faithfully by conjugation on $S$. Thus,
identifying $S$ as a subgroup of $\Aut(S)$, we have $$S\leq G\leq
\Aut(S),$$ for some non-abelian simple group $S$. This is our first reduction.

The classification of finite simple groups has allowed for the
complete classification the finite $2$-transitive groups. For the
reader's convenience we report in Table~\ref{table0} such
classification extracted from~\cite[page~197]{Ca}.

\begin{table}[!h]
\begin{tabular}{ccccc}\hline
Line&Group $S$&Degree&Condition on $G$&  Remarks\\\hline
1&$\Alt(n)$&$n$&$\Alt(n)\leq G\leq \Sym(n)$&$n\geq 5$\\
2&$\PSL_n(q)$&$\frac{q^n-1}{q-1}$&$\PSL_n(q)\leq G\leq \mathrm{P}\Gamma\mathrm{L}_n(q)$&$n\geq 2$, $(n,q)\neq(2,2),(2,3)$\\
3&$\Sp_{2n}(2)$&$2^{n-1}(2^n-1)$&$G=S$&$n\geq 3$\\
4&$\Sp_{2n}(2)$&$2^{n-1}(2^n+1)$&$G=S$&$n\geq 3$\\
5&$\mathrm{PSU}_3(q)$&$q^3+1$&$\mathrm{PSU}_3(q)\leq G\leq \mathrm{P}\Gamma\mathrm{U}_3(q)$&$q\neq 2$\\
6&$\Sz(q)$&$q^2+1$&$\Sz(q)\leq G\leq \Aut(\Sz(q))$&$q=2^{2m+1}$, $m>0$\\
7&$\Ree(q)$&$q^3+1$&$\Ree(q)\leq G\leq \Aut(\Ree(q))$&$q=3^{2m+1}$, $m> 0$\\
8&$M_{n}$&$n$&$M_n\leq G\leq \Aut(M_n)$&$n\in \{11,12,22,23,24\}$,\\
&&&& $M_n$ Mathieu group,\\
&&&& $G=S$ or $n=22$\\
9&$M_{11}$&$12$&$G=S$&\\
10&$\PSL_2(11)$&$11$&$G=S$&\\
11&$\Alt(7)$&$15$&$G=S$&\\
12&$\PSL_2(8)$&$28$&$G=\mathrm{P}\Sigma\mathrm{L}_2(8)$&\\
13&$HS$&$176$&$G=S$&$HS$ Higman-Sims group\\
14&$Co_3$&$276$&$G=S$&$Co_3$ third Conway group\\\hline
\end{tabular}
\caption{Finite $2$-transitive groups of almost simple type}\label{table0}
\end{table}


An easy computation with a computer (using the computer algebra system
\texttt{Magma}~\cite{magma}) allows as to show that Theorem~\ref{main}
holds true for many of the lines in Table~\ref{table0}.
\begin{proposition}\label{reduction2}
  Let $G$ be a $2$-transitive group as in lines~$1$,~$8$--$12$,
  or~$14$ of Table~$\ref{table0}$. Then Theorem~$\ref{main}$ holds for
  $G$.
\end{proposition}
\begin{proof}
  If $G$ is as in line~$1$, then $G=\Alt(n)$ or $G=\Sym(n)$. In both
  cases from~\cite{CaKu,GoMe,MR2302532,LaMa}, every intersecting set
  of maximal size of $G$ is the coset of the stabiliser of a point and
  hence the lemma follows immediately.

  Now for every group $G$ as in lines~8--12 or~14 of
  Table~$\ref{table0}$ we may compute the minimum eigenvalue of
  $\Gamma_G$ using Lemma~\ref{eigenvalues} and the computer algebra
  system \texttt{Magma} (or the character tables in~\cite{ATLAS}). In
  all cases, we see that the minimum eigenvalue is
  $-|\mathcal{D}|/(|\Omega|-1)$, where $\mathcal{D}$ is the set of
  derangements of $G$ and $\Omega$ is the set acted upon by $G$. Thus
  the proposition follows from Lemma~\ref{lemma:7}.

\end{proof}

In particular, Proposition~\ref{reduction2} is our second reduction
and allows us to consider only lines~$2$,~$3$,~$4$,~$5$,~$6$,~$7$
and~$13$ of Table~\ref{table0}.

\begin{proposition}\label{prop:1}Let $G$ be a transitive group on
  $\Omega$ and let $H$ be a transitive subgroup of $G$. Suppose that
  every intersecting set of $H$ has cardinality at most
  $|H|/|\Omega|$, then every intersecting set of $G$ has cardinality
  at most $|G|/|\Omega|$.
\end{proposition}
\begin{proof}
  Let $X$ be an intersecting set of maximal cardinality of $G$. Let
  $R$ be a set of representatives for the right cosets of $H$ in
  $G$. Thus $G=\cup_{r\in R}Hr$ and $X=\cup_{r\in R}(Hr\cap
  X)$. Choose $\bar{r}\in R$ with $|H\bar{r}\cap X|$ as large as
  possible. Observe that
\[
(H\bar{r}\cap X)\bar{r}^{-1}=H\cap X\bar{r}^{-1}
\] 
is an intersecting set of $H$, and hence 
\[
|H\bar{r}\cap X|=|H\cap X\bar{r}^{-1}|\leq |H|/|\Omega|.
\]
In particular, 
\[
|X|=\sum_{r\in R}|Hr\cap X|\leq |R||H\bar{r}\cap X|\leq |R||H|/|\Omega|=|G|/|\Omega|.
\]
\end{proof}

Proposition~\ref{prop:1} offers our third reduction: in proving Theorem~\ref{main} we may assume that $G=S$ is a non-abelian simple group. (In fact, in view of our first reduction $G$ is almost simple with socle the non-abelian simple group $S$. Then, by Table~\ref{table0}, we see that either $S$ itself is $2$-transitive, or $S=\PSL_2(8)$ and $G=\mathrm{P}\Sigma\mathrm{L}_2(8)$. However we have considered the latter possibility in Proposition~\ref{reduction2}.)

%

\section{Suzuki groups: line~$6$ of Table~\ref{table0}}\label{sec:suzuki}

In this section we follow the notation and the information on the
Suzuki groups in~\cite{suzuki}. Let $\ell$ be a positive integer,
$q=2^{2\ell+1}$ and $r=2^{\ell+1}$. Observe that $2q=r^2$. We denote
by $\Sz(q)$ the Suzuki group defined over the finite field with $q$
elements. The group $\Sz(q)$ has order $(q^2+1)q^2(q-1)$ and has a
$2$-transitive action on the points of an inversive plane of
cardinality $q^2+1$. Let $\pi$ be the permutation character of
$\Sz(q)$ and write $\pi=\chi_0+X$ where $\chi_0$ is the principal
character of $\Sz(q)$ and $X$ is an irreducible character of degree
$q^2$. The conjugacy classes and the character table of $\Sz(q)$ are
described in~\cite[Section~$17$]{suzuki}.

Suzuki in~\cite{suzuki} subdivides the irreducible complex characters into six families:
\begin{description}
\item[(1)] the principal character $\chi_0$, 
\item[(2)] the character $X$, 
\item[(3)]the family $(X_i)_{i=1}^{q/2-1}$ of $q/2-1$ characters of degree $q^2+1$, 
\item[(4)]the family $(Y_{j})_{j=1}^{(q+4)/4}$ of $(q+r)/4$ characters of degree $(q-r+1)(q-1)$, 
\item[(5)]the family $(Z_{k})_{k=1}^{(q-r)/4}$ of $(q-r)/4$ characters of degree $(q+r+1)(q-1)$, and 
\item[(6)]the family $(W_1,W_2)$ of two characters of degree $r(q-1)/2$.
\end{description}

The only conjugacy classes of $\Sz(q)$ relevant for our work are the conjugacy classes consisting of derangements. In the work of Suzuki these classes are partitioned into two families: 
\begin{description}
\item[(1)]the family $\pi_1$ consisting of $(q+r)/4$ conjugacy classes each of cardinality $(q-r+1)q^2(q-1)$, and 
\item[(2)]the family $\pi_2$ consisting of $(q-r)/4$ conjugacy classes each of cardinality $(q+r+1)q^2(q-1)$.
\end{description}

The character values are then described in~\cite[Section~$17$ and Theorem~$13$]{suzuki}. Using Lemma~\ref{eigenvalues}, this information and straightforward computations we obtain Table~\ref{table2}, which describes the eigenvalues (together with their multiplicities) of the derangement graph of $\Sz(q)$.

\begin{table}[!h]
\begin{tabular}{cccc}\hline
Eigenvalue&Multiplicity&Comments\\\hline
$\frac{q^3(q-1)^2}{2}$&$1$&Afforded by $\chi_0$\\
$-\frac{q(q-1)^2}{2}$&$q^4$&Afforded by $X$\\ 
$0$&$(q^2+1)^2\frac{q-2}{2}$&Afforded by $(X_i)_{i=1}^{q/2-1}$\\
$q^2$&$\frac{q(q-1)^3(q+1)}{2}$&Afforded by $(Y_i)_{i=1}^{(q+r)/4}$, $(Z_k)_{k=1}^{(q-r)/4}$ and $(W_l)_{l=1}^2$\\\hline
\end{tabular}
\caption{Eigenvalues and multiplicities for the derangement graph of $\Sz(q)$}\label{table2}
\end{table}

We are now ready to prove Theorem~\ref{main} for the $2$-transitive
groups with a Suzuki group as the socle.
\begin{proposition}\label{prop:suzuki}
Let  $G$ be a $2$-transitive group as in line~$6$ of Table~$\ref{table0}$. Then Theorem~$\ref{main}$  holds for $G$.
\end{proposition}
\begin{proof}

  We use the notation that we established above and the reductions in
  Section~\ref{reductios}.  Let $S$ be the socle of $G$. Thus
  $G=S=\Sz(q)$. Then, from Table~\ref{table2}, the graph $\Gamma_G$ has
  a unique negative eigenvalue and hence the proof follows immediately
  from Lemma~\ref{lemma:8}.

\end{proof}


\section{Ree groups: line~$7$ of Table~\ref{table0}}\label{sec:reegroups}

The analysis in this section is similar to the analysis in
Section~\ref{sec:suzuki}. (All the information needed in this section
can be found in~\cite{ward}.)

Let $\ell$ be a positive integer, $q=3^{2\ell+1}$ and $m=3^\ell$. We
denote by $\Ree(q)$ the Ree group defined over the finite field with
$q$ elements. The group $\Ree(q)$ has order $(q^3+1)q^3(q-1)$ and has
a $2$-transitive action of degree $q^3+1$ on the points of a Steiner
system. Let $\pi$ be the permutation character of $\Ree(q)$ and write
$\pi=\xi_1+\xi_3$ where $\xi_1$ is the principal character of
$\Ree(q)$ and $\xi_3$ is an irreducible character of degree $q^3$. The
conjugacy classes and the character table of $\Ree(q)$ are described
in~\cite{ward}.

Since we are only interested in derangements, we only say a few words
on the conjugacy classes of $\Ree(q)$ consisting of
derangements. Ward~\cite{ward} subdivides the conjugacy classes of
derangements into four families:
\begin{description}
\item[(1)] a family consisting of $(q-3)/24$ conjugacy classes each
  having size $(q^2-q+1)q^3(q-1)$ (denoted by $(S^a)_a$);
\item[(2)] a family consisting of $(q-3)/8$ conjugacy classes each
  having size $(q^2-q+1)q^3(q-1)$ (denoted by $(JS^a)_a$);
\item[(3)] a family consisting of $(q-3m)/6$ conjugacy classes each
  having size $(q+1+3m)q^3(q^2-1)$ (denoted by $V$);
\item[(4)] a family consisting of $(q+3m)/6$ conjugacy classes each
  having size $(q+1-3m)q^3(q^2-1)$ (denoted by $W$).
\end{description}

Ward~\cite{ward} subdivides the irreducible complex characters into
several families (namely, $\{\xi_i\}_i$, $\{\eta_r\}_r$,
$\{\eta_r'\}_r$, $\{\eta_t\}_t$, $\{\eta_t'\}_t$, $\{\eta_i^-\}_i$ and
$\{\eta_i^+\}_i$). The value of some of these characters on
derangements is quite tricky to extract from~\cite{ward}: especially
the irreducible characters defined ``exceptional'' by Ward (like
$\{\eta_t\}_t$ and $\{\eta_t'\}_t$). Because of this technical
difficulties, we do not compute all the eigenvalues of the derangement
graph of $\Ree(q)$, but we content ourselves to the eigenvalues
$\lambda(\xi_i)$, where $i\in \{1,\ldots,10\}$. In fact, for these
characters the value $\lambda(\xi_i)$ can be easily computed by simply
reading the character table on page~87 and~88 in~\cite{ward} and by
using the information on the conjugacy classes of derangements of
$\Ree(q)$ that we outlined above (and by invoking Lemma~\ref{eigenvalues}). We sum up this information in
Table~\ref{table3}.

\begin{table}[!h]
\begin{tabular}{ccc}\hline
Eigenvalue&Comment\\\hline
$\frac{q^3(q-1)(q^3-2q^2-1)}{2}$&Afforded by $\xi_1$, valency of $\Gamma_{\Ree(q)}$\\
$-\frac{(q-1)(q^3-2q^2-1)}{2}$&Afforded by $\xi_3$\\ 
$0$&Afforded by $\xi_2$ and $\xi_4$\\
$3mq^2(-q+4m-1)$&Afforded by $\xi_5$ and $\xi_7$\\
$3mq^2(q+4m+1)$&Afforded by $\xi_6$ and $\xi_8$\\
$q^3$&Afforded by $\xi_9$ and $\xi_{10}$\\\hline
\end{tabular}
\caption{Some eigenvalues for the derangement graph of $\Ree(q)$}\label{table3}
\end{table}

\begin{proposition}\label{prop:ree}
Let  $G$ be a $2$-transitive group as in line~$7$ of Table~$\ref{table0}$. Then Theorem~$\ref{main}$  holds for $G$.
\end{proposition}
\begin{proof}
We use the notation and the reductions that we established above.
Let $S$ be the socle of $G$. Thus $G=S=\Ree(q)$. We show that $\lambda(\xi_3)$ is the minimal eigenvalue of $\Gamma_{G}$ and that $\xi_3$ is the unique irreducible character affording this eigenvalue, the proof will then follow immediately from Lemma~\ref{lemma:8}.

We argue by contradiction and we assume that there exists $\chi^*\in \Irr(G)$ with $\lambda(\chi^*)\leq\lambda(\xi_3)$ and $\chi^*\neq \xi_3$. From Lemma~\ref{critical} and from Table~\ref{table3}, we have
\begin{equation}\label{eq:ree}\chi^*(1)\leq (|\Omega|-1)\sqrt{\frac{|G|}{|\mathcal{D}|}-2}=q^3\sqrt{\frac{4q^2+4}{q^3-2q^2-1}}.
\end{equation}
By comparing the character degrees of $\Ree(q)$ (given in~\cite[page~$87$]{ward}) with Eq.~\eqref{eq:ree}, we get 
$\chi^*\in \{\xi_2,\xi_5,\xi_6,\xi_7,\xi_8,\xi_9,\xi_{10}\}$. Now, Table~\ref{table3} yields $\lambda(\chi^*)>\lambda(\xi_3)$, a contradiction.
\end{proof}

\section{Higman-Sims Group: line 13 of Table~\ref{table0}}
\label{sec:HigmanSims}

Let $G$ be the Higman-Sims group: this is a sporadic group with
$44352000$ elements and a $2$-transitive action on a geometry, known
as ``Higman's Geometry'', that has 176 points. Let $\pi = \chi_0 +
\psi$ denote the permutation character for this action and let
$\chi_0$ be the principal character of $G$.  As stated in
Section~\ref{algebraicgraph}, there is an irreducible representation
of $G$ for which the corresponding eigenvalue of the derangement graph
is strictly less than $\lambda(\psi)$. In this case we prove that
Theorem~\ref{main} still holds for $G$ using a weighted adjacency
matrix.

The Higman-Sims group has $5$ conjugacy classes of derangements (the
class \texttt{4B, 8A, 11A, 11B} and \texttt{15A} using the notation
in~\cite{ATLAS}). Weight the two classes \texttt{11A} and \texttt{11B}
with a $1$ and all other classes with a $0$. Now, the eigenvalues of
the weighted adjacency matrix can be directly calculated. The largest
eigenvalue is $8064000$, and the least is $-46080$. By
Lemma~\ref{weightedratio}, if $S$ is a coclique in $\Gamma_{G}$, then
\[
|S| \leq \frac{|G|}{1 - \frac{-46080}{8064000}} = \frac{|G|}{176}.
\]

\section{Projective Special Unitary Groups: Line 5 of Table~\ref{table0}}
\label{sec:PSU}

Let $G$ be the group $\mathrm{PSU}_3(q)$. We follow the notation and use the
character table for $G$ that is given in \cite{MR0335618}.
This group has order $q^3(q^3+1)(q^2-1)/d$ where $d = \gcd(3, q+1)$
and a $2$-transitive action with order $q^3+1$. Let $\pi = \chi_0 +
\psi$ denote the permutation character of this action and $\chi_0$ the principal
character.  The eigenvalue $\lambda(\psi)$ is not the least eigenvalue
for the adjacency matrix of the derangement graph, so we will consider
a weighted adjacency matrix.

For $\gcd(3,q+1) =1$, there are two families of conjugacy classes in
$\mathrm{PSU}_3(q)$ that are derangements, $C_1$ and $C_2$:
\begin{description}
\item[(1)]$C_1$ has $(q^2-q)/3$ classes each of size ${|G|}/{(q^2-q+1)}$, and 
\item[(2)]$C_2$
has $(q^2-q)/6$ classes each of size ${|G|}/{(q+1)^2}$. 
\end{description}For the
values of $q$ with $\gcd(q+1,3) =3$, there are three families of
conjugacy classes of derangements in $\mathrm{PSU}_3(q)$ which we denote by
$C_1$, $C_2'$ and $C_2''$:
\begin{description}
\item[(1)]$C_1$ contains
${(q^2-q-2)}/{9}$ classes each of size ${3|G|}/{(q^2-q+1)}$,  
\item[(2)] $C_2'$ contains ${(q^2-q-2)}/{18}$ classes of size
${3|G|}/{(q+1)^2}$ and
\item[(3)] $C_2''$ contains a single class of size
${|G|}/{(q+1)^2}$. 
\end{description}
Set $C_2 = C_2' \cup C_2''$. 

Define the $(a,b)$-weighted adjacency matrix $A$ as follows. The
$(\sigma,\pi)$-entry of $A$ is equal to $a$ if $\sigma^{-1}\pi$ is in
one of the conjugacy classes in the family $C_1$, and the entry is
equal to $b$ if $\sigma^{-1}\pi$ is in one of the conjugacy classes in
the family $C_2$, and any other entry is $0$. Set
\[
a= \frac{q(2q^2+q-1)}{3|C_1|},
\quad 
b= \frac{q(q^2-q+1)}{3|C_2|},
\]
where (with an abuse of notation) $|C_i|$ denotes the total number of
elements in the conjugacy classes in the family $C_i$ rather than the
size of the family. Thus
\begin{eqnarray*}
|C_1|&=&
\begin{cases}
\frac{q^4(q^2-1)^2}{3}&\textrm{if }\gcd(q+1,3)=1,\\
\frac{q^3(q-1)(q+1)^3(q-2)}{9}&\textrm{if }\gcd(q+1,3)=3,\\
\end{cases}\\
|C_2|&=&
\begin{cases}
\frac{q^4(q-1)^2(q^2-q+1)}{6}&\textrm{if }\gcd(q+1,3)=1,\\
\frac{q^3(q-1)(q^2-q+1)(q^2-q+4)}{18}&\textrm{if }\gcd(q+1,3)=3.\\
\end{cases}
\end{eqnarray*}

Straight-forward calculations (using the table in~\cite{MR0335618})
produces the eigenvalues of $A$ for three of the irreducible characters of
$\mathrm{PSU}_3(q)$ (for all admissible values of $q$).

\begin{lemma}\label{psuevalues-1}
  The eigenvalue of $A$ afforded by the principal character is $q^3$. The
  eigenvalue afforded by  $\psi$ is $-1$. The eigenvalue afforded by the
  irreducible representation with degree $q(q-1)$ is $-1$. 
\end{lemma}

We will calculate the exact value of the eigenvalues of $A$
corresponding the representations of degree $q^2-q+1$. We will
consider the cases where $\gcd(3,q+1)=1$ and $\gcd(3, q+1) = 3$
separately.

\begin{lemma}\label{psuevalues1}
  Assume that $\gcd(3,q+1) = 1$.
  If $q$ is odd, then the eigenvalue for exactly one of the
  irreducible representations with degree $q^2-q+1$ is equal to $-1$,
  for any other irreducible representations with degree $q^2-q+1$, the
  eigenvalue is equal to $2/(q-1)$. If $q$ is even, then the
  eigenvalue for every irreducible representation of degree $q^2-q+1$
  is equal to $2/(q-1)$.
\end{lemma}
\begin{proof}
Define
\[
T=\{ (k,l,m) \, : \, k+l+m \equiv 0 \pmod{q+1}, \, 1\leq k <l<m\leq q+1\}.
\]
For $\gcd(3,q+1)=1$, the conjugacy classes in $C_2$ are parametrised
by triples from the set $T$. The irreducible representations with
degree $q^2-q+1$ are parametrised by $u \in \{1,\dots, q\}$ and we
will denote them by $\chi_u$. The value of $\chi_u$ on the conjugacy
classes in $C_1$ is $0$. The sum of the character $\chi_u$ over all
the conjugacy classes in the family $C_2$ is
\begin{align}\label{eq:chi3}
\sum_{(k,l,m) \in T} e^{3uk}+e^{3ul}+e^{3um},
\end{align}
where $e$ is a complex primitive $(q+1)$th root of unity. 

Simple counting arguments (that we omit) will show the following two results.
\begin{claim}\label{claim1}
If $q$ is odd, then 
\begin{description}
\item[(1)] the element $q+1$ occurs in exactly $(q-1)/2$ triples in $T$,
\item[(2)] any odd element from $\{1,\dots , q\}$ occurs in exactly $(q-1)/2$ triples in $T$,
\item[(3)] any even element from $\{1,\dots , q\}$ occurs in exactly
  $(q-3)/2$ triples in $T$.
\end{description}
\end{claim}

\begin{claim}\label{claim2}
If $q$ is even, then 
\begin{description}
\item[(1)] the element $q+1$ occurs in exactly $q/2$ triples in $T$,
\item[(2)] any element from $\{1,\dots , q\}$ occurs in exactly $(q-2)/2$ triples in $T$.
\end{description}
\end{claim}

If $q$ is odd then there is an irreducible representation parametrised
by $u= (q+1)/2$. In this case, and using Claim~\ref{claim1}, the value
of the sum in Eq.~\eqref{eq:chi3} is $-(q+1)/2$ and by
Eq.~\eqref{eq:weightedevalues} the corresponding eigenvalue is
$-1$. If $u \neq \frac{q+1}{2}$, then, using Claim~\ref{claim1} or
Claim~\ref{claim2} as appropriate, the sum in Eq.~\eqref{eq:chi3}
is $1$. A straight-forward calculation using
Eq.~\eqref{eq:weightedevalues} then shows that the eigenvalue for
this representation is $2/(q - 1)$.  
\end{proof}

\begin{lemma}\label{psuevalues2}
  Assume that $\gcd(3,q+1) = 3$.
  The eigenvalue for irreducible representations with degree $q^2-q+1$
  is equal to $6q/(q^2-q+4)$.
\end{lemma}
\begin{proof}
Define 
\[
T'=\{ (k,l,m) \, : \, k+l+m \equiv 0 \pmod{q+1}, \, 1\leq k <l \leq (q+1)/3, \, \ell<m\leq q+1\}.
\]
The conjugacy classes in family $C_2'$, when $\gcd(3,q+1)=3$, are
parametrised by a triple from $T'$.

Let $\chi_u$ be an irreducible representation with degree $q^2-q+1$;
these characters are parameterized by $u \in \{1,\dots,
(q+1)/3-1\}$. The value of $\chi_u$ on the conjugacy classes of type
$C_1$ is $0$ and the value of $\chi_u$ on the conjugacy class of type
$C_2''$ is $3$.  The sum of $\chi_u$ over all conjugacy classes of type
$C_2'$ is
\begin{align}\label{eq:chi1divides}
\sum_{(k,\ell,m) \in T'} e^{3ku}+e^{3\ell u}+e^{3m u},
\end{align}
(here $e$ is a complex primitive $(q+1)$th root of unity). 

The exact value of the sum in Eq.~\eqref{eq:chi1divides} can be
determined, we will do this first for $q$ odd and then for $q$ even.
If $q$ is odd, then for every $i$ there are, in total, $\frac{q-2}{6}$
elements $j$ in triples from $T'$ with $j \equiv 3i \pmod{q+1}$. In
this case we have that the sum in Eq.~\eqref{eq:chi1divides} is equal
to
\begin{align*}
\frac{q-2}{3} \sum_{i=1}^{\frac{q+1}{3}}e^{3i} + \frac{q-2}{6} \sum_{i=1}^{\frac{q+1}{3}}e^{3i} = 0.
\end{align*}
If both $q$ and $i$ are even, then there are in total $\frac{q-5}{6}$
elements $j$ in triples from $T'$ for which $j \equiv 3i \pmod{q+1}$.
If $q$ is even and $i$ is odd, then there are in total $\frac{q+1}{6}$
elements $j$ in triples from $T'$ for which $j \equiv 3i
\pmod{q+1}$. Thus for $q$ even the sum in Eq.~\eqref{eq:chi1divides}
is
\begin{align*}
\frac{q-2}{3} \sum_{i=1}^{\frac{q+1}{3}}e^{3i} 
 + \frac{q-5}{6} \sum_{i=1}^{\frac{q+1}{6}}e^{6i}
 + \frac{q+1}{6} \sum_{i=1}^{\frac{q+1}{6}}e^{3(2i-1)}
= 0.
\end{align*}

A straight-forward application of Eq.~\eqref{eq:weightedevalues}
shows that the eigenvalue for these representations is $6q/(q^2-q+4)$.
\end{proof}

\begin{proposition}
  Let $G =\mathrm{PSU}_3(q)$, then Theorem~\ref{main} holds.
\end{proposition}
\begin{proof} Let $A$ be the $(a,b)$-weighted adjacency matrix defined above.
The eigenvalues of $A$ can be directly calculated for $q \leq 5$, so
we will assume that $q\geq 6$.  From Lemma~\ref{psuevalues-1},
$\lambda(\psi) =-1$, we will show that this is the least eigenvalue of
$A$.

If $\chi$ is an irreducible representation with $\lambda(\chi) \leq \lambda(\psi)$, then by Lemma~\ref{weightedcritical}
\begin{align*}
\chi(1) &\leq  (|\Omega|-1) \sqrt{ |G| \frac{\sum_{i=0}^{\ell} a_i^2 |C_i|}{(\sum_{i = 1}^\ell a_i |C_i|)^2} -2 }\\
&= (q^3) \sqrt{|G| \frac{  (\frac{q(2q^2+q-1)}{3})^2 \frac{1}{|C_1|} + (\frac{q(q^2-q+1)}{3})^2\frac{1}{|C_2|}} 
                    {(\frac{q(2q^2+q-1) + q(q^2-q+1)}{3})^2}   - 2 }. 
\end{align*}
Some rather tedious, but not complicated, calculation show that this
is strictly smaller than $(q-1)(q^2-q+1)$ for $q\geq 6$ (the cases
where $\gcd(3,q+1)$ is equal to $1$ and $3$ need to be considered
separately).

The only irreducible representations of $\mathrm{PSU}_3(q)$ with degree less
than $(q-1)(q^2-q+1)$ are the representations with degree $q(q-1)$ and
$q^2-q+1$ (this is from \cite[Table 2]{MR0335618}). We have calculated
in Lemmas~\ref{psuevalues-1}, \ref{psuevalues1} and \ref{psuevalues2}
that the eigenvalue for any of these representations is one of $-1$,
$2/(q-1)$ or $6q/(q^2-q+4)$. This implies that the least eigenvalue of
$A$ is $-1$ and by Lemma~\ref{weightedratio} and intersecting set
in $G$ is no larger than $|G|/(q^3+1)$.  
\end{proof}


\section{Projective linear groups: line~$2$ of Table~\ref{table0}}\label{sec:PSL}

By the reductions given in Section~\ref{reductios}, we only need to
consider the groups $\PSL_n(q)$. Throughout this section, let $G = \PSL_n(q)$. This group has
$\frac{\Pi_{i=1}^{n-1}(q^n-q^i)}{d(q-1)}$ elements, where $d=
\gcd(n,q-1)$, and a natural action of order $\frac{q^n-1}{q-1}$ on the points of the
projective space. The first step in this proof is to show that the
proportion of derangements in $\PSL_n(q)$ is large.

Given a positive integer $n$, we denote by $\varphi(n)$ the {\em Euler
  totien function}.
\begin{lemma}\label{lemma:1}Let $n$ be a positive integer with
  $n>6$. Then $\varphi(n)\geq n/\log_2(n)$, and $\varphi(n)\geq
  2n/(\log_3(n)+2)$ if $n$ is odd.
\end{lemma}

\begin{proof}
  Write $n=2^{\alpha_0}p_1^{\alpha_1}\cdots p_\ell^{\alpha_\ell}$
  where $\alpha_0,\ell\geq 0$ and $p_1,\ldots,p_\ell$ are distinct odd
  primes with $p_1<p_2<\cdots <p_\ell$ and
  $\alpha_1,\ldots,\alpha_\ell\geq 1$. Also write
  $m=n/2^{\alpha_0}$. Observe that $\ell\leq \log_3(m)$.

Assume first that $\ell=0$, that is, $n=2^{\alpha_0}$ is a $2$-power.
Observe that $\alpha_0\geq 3$ because $n>6$. Then $\varphi(n)=n/2\geq n/\alpha_0=n/\log_2(n)$. Assume next that $n$ is odd, that is, $\alpha_0=0$ and $n=m$. Now,
\begin{eqnarray*}
\varphi(n)&=&n\left(1-\frac{1}{p_1}\right)\left(1-\frac{1}{p_2}\right)\cdots \left(1-\frac{1}{p_\ell}\right)\geq n\cdot \frac{2}{3}\cdot\frac{3}{4}\cdots\frac{\ell+1}{\ell+2}\\
&=&\frac{2n}{\ell+2}\geq \frac{2n}{\log_3(n)+2}\geq \frac{n}{\log_2(n)},
\end{eqnarray*}
where the last inequality follows from an easy computation.

Finally assume that $\alpha_0>0$ and $\ell>0$. Now,
\begin{eqnarray*}
\varphi(n)&=&2^{\alpha_0-1}\varphi(m)\geq 2^{\alpha_0}\frac{m}{\log_3(m)+2}=\frac{n}{\log_3(m)+2}.
\end{eqnarray*}
If $\alpha_0\geq 2$, then $\log_3(m)+2\leq
\log_2(m)+\alpha_0=\log_2(n)$. Suppose that $\alpha_0=1$.  Now, if
$n\geq 14$, then we get $\log_3(m)+2=\log_3(n/2)+2\leq \log_2(n)$ and
the lemma follows. If $n=10$, then the lemma follows with a direct
computation.
\end{proof}

Let $n$ be a positive integer and let $q$ be a prime power.
\begin{lemma}\label{lemma:2}
The proportion of derangements in $\PSL_n(q)$ is at least $\frac{1}{n^2\log_2(q)}$.
\end{lemma}
\begin{proof}

  In the next paragraph we recall some information on the conjugacy
  classes of $\PSL_n(q)$, which can be found  for instance in~\cite{C1,C2}.

  Let $p$ be the proportion of derangements of $G$ in its natural
  action on the projective space and let $C$ be a Singer cycle of
  $G$. Now, $C$ is a cyclic group of order $\frac{q^n-1}{d(q-1)}$,
  where $d=\gcd(n,q-1)$. Moreover, $C=\cent GC$, $|\norm G C:\cent G
  C|=n$ and $\norm G C/\cent G C$ is cyclic and generated by an
  element of the Weyl group of $G$. Every non-identity element of $C$
  acts fixed-point-freely. Moreover, for every $x\in C$ with
  $C=\langle x\rangle$, we have $C=\cent G x$.

  Fix $$D=\bigcup_{\substack{x\in C\\C=\langle x\rangle}}x^G.$$ Now
  $p\geq |D|/|G|$ and, from the previous paragraph, $D$ is the union
  of at least $\varphi(|C|)/n$ $G$-conjugacy classes and each conjugacy
  class has size $|G:C|$. If $|C|>6$, using Lemma~\ref{lemma:1}, we
  obtain
\begin{eqnarray*}
p&\geq& \frac{\frac{\varphi(|C|)}{n}\cdot |G:C|}{|G|}\geq\frac{1}{n\log_2(|C|)}\\
&\geq&\frac{1}{n\log_2((q^n-1)/(q-1))}\geq\frac{1}{n\log_2(q^n)}=\frac{1}{n^2\log_2(q)}.
\end{eqnarray*}

If $|C|\leq 6$, then $n=2$ and $q\leq 11$. For each of these groups we
can check (with a case-by-case explicit computation) that the
statement of the lemma holds.
\end{proof}
In light of recent results by Fulman and
Guralnick~\cite{G1,G2,G22,G3}, Lemma~\ref{lemma:2} is rather weak (but
still suitable for our application).  In fact, answering a conjecture
due independently to Shalev and Boston et al.~\cite{B}, Fulman and
Guralnick have proved that there exists a constant $C>0$ such that, in
every non-abelian simple transitive permutation group, the proportion
of elements which are derangements is at least $C$. In particular, in
view of this remarkable theorem, in Lemma~\ref{lemma:2} one might
replace the function $1/n^2\log_2(q)$ by the constant $C$. However,
even for rather natural transitive actions, currently there is no good
estimate on $C$: in fact, we are not aware of any bound uniform in $q$
and $n$ for the proportion of derangements in $\PSL_n(q)$ in its
natural action on the projection space. Following the arguments in the
work of Fulman and Guralnick this seems possible to achieve (at least
for $q\geq 7$), but it would take us too far astray to do it here.

\begin{proposition}\label{prop:PSL}Let $G$ be a $2$-transitive group as in line~$2$ of Table~$\ref{table0}$. Then Theorem~$\ref{main}$ holds for $G$.
\end{proposition}
\begin{proof}
From our preliminary reductions we may assume that $G=\PSL_n(q)$ is endowed of its natural action on the points of the projective space. Moreover, from the work in~\cite{KaPa,KaPa2}, we may assume that $n\geq 4$.

Let $\pi$ be the permutation character of $G$ and write $\pi=\chi_0+\psi,$ where $\chi_0,\psi$ are the irreducible constituents of $\pi$ and $\chi_0$ is the principal character of $G$.  We show that $\lambda(\psi)$ is the minimal eigenvalue of $\Gamma_G$ and $\psi$ is the only irreducible complex character of $G$ realising this minimum. We argue by contradiction and we assume that there exists $\chi^*\in \Irr(G)$ with $\chi^*\neq \psi$ and with $\lambda(\chi^*)\leq \lambda(\psi)$. From Lemmas~\ref{critical} and~\ref{lemma:2}, we have
\begin{equation}\label{eq:PSLPSL}
\chi^*(1)\leq\frac{q^n-q}{q-1}\sqrt{n^2\log_2(q)-2}.
\end{equation}

We now use some character-theoretic results of the third author and Zalesskii~\cite{TZ}. For the reader's convenience we reproduce the relevant result from~\cite{TZ} when $n\geq 4$. 

Denote by $1=d_0<d_1<d_2<\ldots<d_\ell$ the character degrees of $G$ and by $N_j$ the number of irreducible complex representations of $G$ (up to equivalence) of degree $d_j$. (The definition of equivalence in this context is in~\cite{TZ} or in~\cite{GT1}.) The value of $(d_i,N_i)$ for $i=1,2,3$ are in Table~\ref{table1}.
\begin{table}[!h]
\begin{tabular}{cccc}\hline
$(d_1,N_1)$&$(d_2,N_2)$&$(d_3,N_3)$&Condition\\\hline
$\left(\frac{q^4-q}{q-1},1\right)$&$\left(\frac{q^4-1}{q-1},q-2\right)$&$(\frac{1}{2}(q^3-1)(q-1),2)$&$n=4,2\nmid q$, $q\neq 3$\\
$\left(\frac{q^4-q}{q-1},1\right)$&$\left(\frac{q^4-1}{q-1},q-2\right)$&$((q^3-1)(q-1),q/2)$&$n=4,2\mid q$, $q\neq 2$\\
$(7,1)$&$(8,1)$&$(14,1)$&$(n,q)=(4,2)$\\
$(26,2)$&$(39,1)$&$(40,1)$&$(n,q)=(4,3)$\\

$\left(\frac{q^n-q}{q-1},1\right)$&$\left(\frac{q^n-1}{q-1},q-2\right)$&$\left(\frac{(q^n-1)(q^{n-1}-q^2)}{(q-1)(q^2-1)},1\right)$&$n\geq 5, q\geq 3, (n,q)\neq (6,3)$\\
$\left(\frac{q^n-q}{q-1},1\right)$&$\left(\frac{(q^n-1)(q^{n-1}-q^2)}{(q-1)(q^2-1)},1\right)$&$\left(\frac{(q^n-1)(q^{n-1}-1)}{(q-1)(q^2-1)},1\right)$&$n\geq 5,n\neq 6,q=2$\\

$(62,1)$&$(217,1)$&$(588,1)$&$(n,q)=(6,2)$\\
$(363,1)$&$(364,1)$&$(6292,2)$&$(n,q)=(6,3)$\\\hline
\end{tabular}
\caption{Small character  degrees of $\PSL_n(q)$}\label{table1}
\end{table}

Comparing Eq.~\eqref{eq:PSLPSL} with Table~\ref{table1}, we are left with one of the following cases:
\begin{description}
\item[(1)] \label{en:1} $q>2$, $n\geq 5$, $\chi^*(1)=d_2$,
\item[(2)]\label{en:2} $q=2$, $n\in \{5,6\}$,
\item[(3)]\label{en:3} $q\geq 16$, $q\neq 17$, $n=4$ and $\chi^*(1)=d_2$,
\item[(4)]\label{en:4} $q\in \{2,3,4,5,7,8,9,11,13,17\}$ and $n=4$.
\end{description}
Now, with the computer algebra system \texttt{Magma}~\cite{magma}, we can check directly that in Cases~(2) and~(4) the minimum eigenvalue is $\lambda(\psi)$ and is realised only by $\psi$. It remains to consider Cases~(1) and~(3). Here, $\chi^*$ is one of the irreducible characters of $G$ of degree $(q^n-1)/(q-1)$. These characters are named Weil characters, after the pioneering work of Andr\'{e} Weil in~\cite{Weil}. An extensive study of Weil characters has begun in~\cite{Gerardin} and now the value of each  Weil character in each conjugacy class of $G$ is explicitly known. For example it is given in the account of Guralnick and the third author (see Eq.~$(1)$ on page~$4976$ in~\cite{GT}.) From this formula it is immediate to see that $\lambda(\chi^*)=0$, contradicting the fact that $\lambda(\chi^*)\leq \lambda(\psi)$.
\end{proof}
The proof of Proposition~\ref{prop:PSL} can be slightly shortened and simplified using a suitable weighted adjacency matrix and using the method we present in the next section for dealing with the Symplectic groups.

\section{Symplectic groups: lines~$3$ and~$4$ of Table~\ref{table0}}\label{Sec:symplectic}

The proof of Theorem~\ref{main} for the two $2$-transitive actions of
the symplectic group $\Sp_{2n}(2)$ is similar to a combination of the
proofs in the projective linear  and  unitary cases. The
investigations of Guralnick and the third author in~\cite{GT} on the
irreducible complex characters of $\Sp_{2n}(2)$ of ``small''
degree will be crucial to this proof. Details about the symplectic
group can be found in \cite{C1,C2}. We start by setting some notation. 

We fix $n$ to be a natural number with $n\geq 3$. We let
$G=\Sp_{2n}(2)$ and we let $V$ be the $2n$-dimensional vector space
over the finite field $\mathbb{F}_2$.  (Thus $V$ is the natural
module for $G$.) The group $G$ has two natural $2$-transitive
actions. Both actions can be seen by viewing $\Sp_{2n}(2)=\Omega_{2n+1}(2)$,
and recalling that $\Omega_{2n+1}(2)$ has a natural action on the
non-degenerate hyperspaces of plus type (which we denote by
$\Omega^+$) and on the non-degenerate hyperspaces of minus type (which
we denote by $\Omega^-$). We have
\begin{equation*}
|\Omega^+|=2^{n-1}(2^n+1), \qquad|\Omega_n^-|=2^{n-1}(2^n-1).
\end{equation*} Moreover, the stabiliser of an element of $\Omega_n^+$ is $H^+=\mathrm{O}_{2n}^+(2)$ and the stabiliser of an element of $\Omega_n^-$ is $H^-=\mathrm{O}_{2n}^-(2)$. 

Given $\varepsilon\in \{+,-\}$, we denote by $\pi^\varepsilon$ the permutation character of $G$ in its action on $\Omega^\varepsilon$ and by $\mathcal{D}^\varepsilon$ the set of derangements of $G$ in its action on $\Omega^\varepsilon$. We write 
\begin{equation}\label{rhopm}
\pi^+=\chi_0+\rho^+, \qquad \pi^-=\chi_0+\rho^-,
\end{equation} where $\chi_0$ is the principal character of $G$.
(Later in our arguments we will be using~\cite{GT}; hence we point
out that the characters $\rho^+$ and $\rho^-$ are studied
in~\cite[Section~6]{GT} in which $\rho^+$ is
denoted by $\rho_n^2$, and $\rho^-$ is denoted by $\rho_n^1$.)

We recall that
$\chi_0+\rho^++\rho^-$
is the permutation character for the transitive action of $G$ on the non-zero elements of $V$ (see for example~\cite{GT}, or~\cite{mePa} for a combinatorial proof), and hence 
\begin{equation}\label{rhoppmm}
(\rho^++\rho^-)(g)=2^{\dim\Ker(g-1_V)}-2
\end{equation}
for every $g\in G$.

In the next paragraph we recall some information on the conjugacy
classes of $G=\Sp_{2n}(2)$, we again refer the reader to~\cite{C1,C2}.

If $\varepsilon = +$, then $\varepsilon \cdot 1$ represents $1$;
similarly, if $\varepsilon = -$, then $\varepsilon \cdot 1$ represents
$-1$. Embedding 
$$T^\varep = \Omega^{-\varep}_2(2^n) < {\mathrm {SL}}_2(2^n) \cong {\mathrm {Sp}}_2(2^n)$$ 
naturally in $G = {\mathrm {Sp}}_{2n}(2)$ (by identifying
$W = \bbF_{2^n}^2$ with $V = \bbF_2^{2n}$), we see that 
$G$ contains a cyclic maximal torus $T^\varepsilon$ of order
$2^n+\varepsilon\cdot 1$. 
Fix a generator $x^\varepsilon$ of $T^\varepsilon$. We claim  that
that $x^\varepsilon$ is contained in a unique conjugate of $H^{-\varep}$ and
$x^\varepsilon\in \mathcal{D}^\varepsilon$; equivalently,
\begin{equation}\label{value}
  \rho^\varep(x^\varep) = -1,\qquad\rho^{-\varep}(x^\varep) = 0.
\end{equation}  
This can be seen as follows. By construction, $x^\varep$ has no nonzero 
fixed points on $W$ and so on $V$ as well. Hence Eq. \eqref{rhoppmm} implies that 
$\rho^+(x^\varep) + \rho^-(x^\varep) = -1$. On the other hand, $\pi^+(x^\varep) = 1+\rho^+(x^\varep)$ is the number of 
$x^\varep$-fixed points on $\Omega^+$, and so ${\mathbb Z} \ni \rho^+(x^\varep) \geq -1$; similarly, 
${\mathbb Z} \ni \rho^-(x^\varep) \geq -1$. It follows that 
$$\{\rho^+(x^\varep),\rho^-(x^\varep)\} = \{0,-1\}.$$ 
Also note that if $\varep = -$ then 
$T^-=\Omega^+_2(2^n)$ can be embedded in $H^+=\Omega^+_{2n}(2)$ again by base change, so
$\rho^+(x^-) \neq -1$. Next, for $n > 3$ let $p$ be a {\it primitive prime divisor} of 
$2^{2n}-1$ (cf. \cite{Zs}), so that $p\mid(2^n+1) = |x^+|$ but $p \nmid |H^+|$. It follows 
for $n > 3$ that $x^+$ cannot be contained in any conjugate of $H^+$, whence 
$\rho^+(x^+) = -1$. The same holds for $n = 3$ by inspecting \cite{ATLAS}.  Consequently,
Eq. \eqref{value} holds for all $n \geq 3$.

We also note that
$T^\varepsilon= \bfC_G(T^\varep) = \cent G{x^\varepsilon}$ and hence 
$|\cent G{x^\varepsilon}|=2^n+\varepsilon\cdot 1$.

Let $A^\varepsilon$ be the square matrix indexed by the elements of $G$ with
\[
(A^\varepsilon)_{g,h}=
\begin{cases}
1&\textrm{if }gh^{-1}\in (x^\varepsilon)^G,\\
0&\textrm{if }gh^{-1}\notin (x^\varepsilon)^G.
\end{cases}
\]
As $(x^\varepsilon)^G\subseteq \mathcal{D}^\varepsilon$, we see that $A^\varepsilon$ is a weighted adjacency matrix for the derangement graph of $G$ in its action on $\Omega^\varepsilon$.

\begin{proposition}\label{prop:Symp}
Let $G$ be a $2$-transitive group as in line~$3$ or~$4$ of Table~$\ref{table0}$. Then Theorem~$\ref{main}$ holds for $G$.
\end{proposition}
\begin{proof}
We can check with a computer that this holds when $n\leq 6$. Therefore we assume that $n\geq 7$.

We use the notation that we have established above.  Moreover, given $\chi\in \Irr(G)$, we denote by $\lambda(\chi,\varepsilon)$ the eigenvalue of $A^\varepsilon$ afforded by $\chi$ (see Eq.~\eqref{eq:weightedevalues}). We determine the maximum and the minimum eigenvalue of $A^\varepsilon$. Since all the weights are non-negative real numbers, the largest eigenvalue $d^\varepsilon$ of $A^\varepsilon$ is realised by the principal character of $G$. Therefore 
\begin{equation}\label{eq:maximumSp}
d^\varepsilon=\lambda(\chi_0,\varepsilon)=|(x^\varepsilon)^G|=\frac{|G|}{|\cent G {x^\varepsilon}|}=\frac{|G|}{2^n+\varepsilon\cdot 1}.
\end{equation}

Let $\tau^\varepsilon$ be the minimum eigenvalue of $A^\varepsilon$. We prove that $\tau^\varepsilon=\lambda(\rho^\varepsilon,\varepsilon)$ and that $\rho^\varepsilon$ is the unique character affording $\tau^\varepsilon$.  We argue by contradiction and we assume that there exists $\chi^*$ with $\chi^*\neq \rho^\varepsilon$ and $\lambda(\chi^*,\varepsilon)\leq \lambda(\rho^\varepsilon,\varepsilon)$.

Lemma~\ref{weightedcritical} gives
\begin{eqnarray}\label{eq:5very}
\chi^*(1)&\leq& (|\Omega^\varepsilon|-1)\sqrt{|G|\frac{|(x^\varepsilon)^G|}{|(x^\varepsilon)^G|^2}-2}\\\nonumber
&=&(|\Omega^\varepsilon|-1)\sqrt{|\cent G{x^\varepsilon}|-2}=(|\Omega^\varepsilon|-1)\sqrt{2^n+\varepsilon\cdot 1-2}.
\end{eqnarray}
The main result of Guralnick and the third author~\cite[Theorem~$1.1$ and Table~$1$]{GT} shows that there exist  irreducible characters  $\alpha_n$, $\beta_n$ and $\zeta_n^1$ with  
$$\alpha_n(1)=\frac{(2^{n-1}-1)(2^n-1)}{3},~\beta_n(1)=\frac{(2^{n-1}+1)(2^n+1)}{3},\zeta_n^1(1)=\frac{2^{2n}-1}{3}
$$ such that, if 
$\chi\in \Irr(G)$ and $\chi(1)>1$, then either
\begin{description} 
\item[(1)] $\chi\in \{\rho^+,\rho^-,\alpha_n,\beta_n,\zeta_n^1\}$, or
\item[(2)] $\chi(1)\geq ((2^{n-1}+1)(2^{n-2}-2)/3-1)2^{n-2}(2^{n-1}-1)$.
\end{description}
We have
$$\left(\frac{(2^{n-1}+1)(2^{n-2}-2)}{3}-1\right)2^{n-2}(2^{n-1}-1)>(|\Omega^\varepsilon|-1)\sqrt{2^n+\varepsilon\cdot 1-2}$$
and hence, by Eq.~\eqref{eq:5very}, we have $\chi^*\in \{\rho^+,\rho^-,\alpha_n,\beta_n,\zeta_n^1\}$. 

From Eq.~\eqref{eq:weightedevalues}, we have
\begin{equation}\label{minimumSp}
\lambda(\rho^\varepsilon,\varepsilon)=\frac{-|(x^\varepsilon)^G|}{|\Omega^\varepsilon|-1}=-\frac{|G|}{(2^n+\varepsilon\cdot 1)(|\Omega^\varepsilon|-1)}.
\end{equation}
Now, Eq.~\eqref{value} yields $\lambda(\rho^{-\varepsilon},\varepsilon)\geq 0>\lambda(\rho^\varepsilon,\varepsilon)$ and hence $\chi^*\in \{\alpha_n,\beta_n,\zeta_n^1\}$.

Recall that $V=\mathbb{F}_2^{2n}$ is the natural module for
$G$. Observe that by taking the tensor product
$V\otimes_{\mathbb{F}_2}\mathbb{F}_4$ we may view $V$ as a
$2n$-dimensional vector space over $\mathbb{F}_4$. Let $\xi$ be a
primitive third root of unity in $\mathbb{F}_4$ and let $\bar{\xi}$ be
a primitive third root of unity in $\mathbb{C}$. Clearly, $1$, $\xi$
and $\xi^2$ are not eigenvalues of the matrices $x^+$ and $x^-$ in
their action on $V$ and hence $\Ker
(x^\varepsilon-1_V)=\Ker(x^\varepsilon-\xi\cdot
1_V)=\Ker(x^\varepsilon-\xi^2\cdot 1_V)=0$. Now it follows
from~\cite[page 4997, Eq.~(4)]{GT} that
\begin{equation}\label{eq:zetazeta}
\zeta_n^1(x^\varepsilon)=\frac{1}{3}(1+\bar{\xi}+\bar{\xi}^2)=0.
\end{equation}
Therefore from Eq.~\eqref{eq:weightedevalues} we have $\lambda(\zeta_n^1,\varepsilon)=0>\lambda(\rho^\varepsilon,\varepsilon)$ and hence $\chi^*\in \{\alpha_n,\beta_n\}$.

Consider the map $\zeta_n:G\to \mathbb{C}$ defined by $\zeta_n(g)=(-2)^{\dim\Ker(g-1_V)}$. It turns out (see~\cite[pages~4976,~4977]{GT}) that $\zeta_n$ is a character of $G$ and 
\begin{equation}\label{weil1}
  \zeta_n=\alpha_n+\beta_n+2\zeta_n^1.
\end{equation}  
Next, consider the map $\kappa_\varepsilon:\mathrm{O}_{2n}^\varepsilon(2)\to
\{-1,1\}$ defined by $\kappa_\varepsilon(g)=(-1)^{\dim
  \Ker(g-1_V)}$. Now, $\kappa_\varepsilon$ is actually a homomorphism
(see for example~\cite[page~xii]{ATLAS}) with kernel the index $2$
subgroup $\Omega_{2n}^\varepsilon(2)$ of
$\mathrm{O}_{2n}^\varepsilon(2)$, and hence $\kappa_\varepsilon$ is a
character of $H^\varep = \mathrm{O}_{2n}^\varepsilon(2)$. It is shown
in~\cite[Eq. (11)]{GT} that
\begin{equation}\label{weil2}
  {\mathrm {Ind}}^G_{H+}(\kappa_+) + {\mathrm {Ind}}^G_{H^-}(\kappa_-) = \zeta_n.
\end{equation}  
Note that
$$\alpha_n(1)+\beta_n(1) < \alpha_n(1) + \zeta_n^1(1) = [G:H^-] < [G:H^+] = \beta_n(1) + \zeta_n^1(1) < 2\zeta^1_n(1).$$
Together with Eqs. \eqref{weil1} and \eqref{weil2}, this implies that
\[
\mathrm{Ind}^G_{H^+}(\kappa_+)=\beta_n+\zeta_n^1,
\qquad
\mathrm{Ind}^G_{H^-}(\kappa_-)=\alpha_n+\zeta_n^1.
\]
Certainly, $\mathrm{Ind}^G_{H^+}(\kappa_+)$ is zero on every element not conjugate to an element of $H^+$ (and hence on $\mathcal{D}^+$), and similarly 
that $\mathrm{Ind}^G_{H^-}(\kappa_-)$ is zero on every element not conjugate to an element of $H^-$ (and hence on $\mathcal{D}^-$). It follows from Eq. \eqref{value} that 
$$(\alpha_n+\zeta_n^1)(x^-)=(\beta_n+\zeta_n^1)(x^+)=0$$ and hence, from Eq.~\eqref{eq:zetazeta}, we obtain
\begin{equation}\label{eq:alphaalpha}
\alpha_n(x^-)=0, \qquad \beta_n(x^+)=0.
\end{equation}
Therefore,  $\lambda(\alpha_n,-)=\lambda(\beta_n,+)=0>\lambda(\rho^\varepsilon,\varepsilon)$ and hence $\chi^*=\alpha_n$ when $\varepsilon=+$ and $\chi^*=\beta_n$ when $\varepsilon=-$.

Now, $\zeta_n(x^\varepsilon)=(-2)^{\dim\Ker(x^\varepsilon-1_ V)}=(-2)^0=1$ and hence $(\alpha_n+\beta_n)(x^\varepsilon)=1$ by Eq.~\eqref{eq:zetazeta}. Thus, Eq.~\eqref{eq:alphaalpha} gives
\[
\alpha_n(x^+)=1, \qquad\beta_n(x^-)=1.
\]
Therefore, $\lambda(\alpha_n,+),\lambda(\beta_n,-)>0>\lambda(\rho^\varepsilon,\varepsilon)$. This finally contradicts the existence of $\chi^*$.

Let $S$ be an independent set of the derangement graph of $G$ in its action on $\Omega^\varepsilon$. Now, using Lemma~\ref{weightedratio} and Eqs.~\eqref{eq:maximumSp} and~\eqref{minimumSp}, we obtain
$$|S|\leq |G|\left(1-\frac{d^\varepsilon}{\tau^\varepsilon}\right)^{-1}=\frac{|G|}{|\Omega^\varepsilon|}$$
and the proposition is proven.
\end{proof}


\end{document}